\documentclass[npg]{copernicus}

\usepackage{amsmath,amsthm,amssymb}
\usepackage{calc}
\usepackage{natbib, color, graphicx}

\newcommand{\reels}{\mathbb R}

\newcommand{\Esp}{\mathbb E}

\begin{document}

\title{Detecting spatial patterns with the cumulant function.\\
Part I: The theory}

\author[1]{Alberto Bernacchia}
\author[2]{Philippe Naveau}
%\author[]{NAME}

\affil[1]{Dipartimento di Fisica E.Fermi, Universita' La Sapienza, Roma, Italy}
\affil[2]{Laboratoire des Sciences du Climat et de l'Environnement, IPSL-CNRS,  France}

%% The [] brackets identify the author to the corresponding affiliation, 1, 2, 3, etc. should be inserted.

\runningtitle{Detecting spatial patterns with the cumulant function}

\runningauthor{Bernacchia and Naveau}

\correspondence{Alberto Bernacchia\\ ({\tt a.bernacchia@gmail.com})}

\received{}
\pubdiscuss{} %% only important for two-stage journals
\revised{}
\accepted{}
\published{}

%% These dates will be inserted by the Publication Production Office during the typesetting process.

\firstpage{1}

\maketitle

\begin{abstract}
In climate studies, 
detecting spatial patterns
 that largely deviate from the sample mean 
 still remains  a statistical challenge.
 Although  a Principal Component Analysis (PCA), or equivalently a Empirical Orthogonal Functions (EOF) decomposition, is often applied on this purpose, it
 can only provide  meaningful results   if the underlying multivariate  distribution is Gaussian.
Indeed, PCA is based on optimizing second order moments quantities  and the covariance matrix can only  capture the full dependence structure  for multivariate Gaussian vectors. 
Whenever the application at hand can not satisfy this normality hypothesis (e.g. precipitation data), alternatives and/or improvements to  PCA  have to be developed and studied. 

To go beyond this second order statistics constraint that limits the applicability of the PCA, we take advantage of the cumulant function that can produce higher order moments information. 
This cumulant function, well-known in the statistical literature,  allows us to propose  a new, simple and fast procedure to identify spatial patterns for non-Gaussian data. 
Our algorithm consists in maximizing the cumulant function.
To illustrate our approach, its implementation for which explicit computations are obtained is performed on three family of of  multivariate random vectors. 
In addition, we show that our algorithm corresponds to selecting the directions along
which projected data display the largest spread over the marginal probability density tails. 
\end{abstract}

%% only used for copernicus2.cls
%\abstract{
% TEXT
\section{Introduction}
In geosciences,  
Principal Component Analysis (PCA) has been  an essential  and powerful tool at detecting
spatial structures amongst  time series recorded at different locations. 
PCA aims at building a decorrelated
representation of the data and it finds the spatial patterns  responsible for
the largest proportion of variability \cite{Rencher98}. PCA can also be viewed as 
a dimensionality reduction technique that extracts the
most relevant components of data, i.e. the ones that maximize the variances. 
Hence, second order moments are the foundation of PCA. 
But relying exclusively on second moments implies that PCA is only optimal when applied to multivariate Gaussian vectors. 
Although rarely stated and even more rarely checked, this underlined normality assumption is not always satisfied in practice. 

Recently, different approaches have been tested to extend the applicability of PCA in geosciences.
In particular,  
NonLinear PCA (NLPCA) has  been
 applied to several geophysical datasets
  \citep[e.g.,][]{Hsieh04,Monahan01}. In the NLPCA algorithm, data are considered as the input of an
 auto-associative neural network with five layers, with a bottleneck
 in the third layer \cite{Kramer91}. 
Through the minimization of a cost function, the output is forced to
 be as close as possible  to the input, and the bottleneck layer
 is a low dimensional representation of the input. Since this neural network
 is nonlinear, NLPCA goes automatically beyond correlations.
However, NLPCA suffers the intrinsic limitations of multilayered
networks \citep[e.g.,][]{Christiansen05,Malthouse98}: it
is computationally expensive and does not always  converge to a global
solution.  
To overcome these difficulties, we pursue a less
ambitious aim: instead of trying to find decompositions that can explain  
 the entire body of data with respect to a criterion, we focus on
the part of data responsible for large anomalous behaviors.

In contrast to PCA, our approach tends  to give maximal weight to data
 points which largely deviate from the mean, and to find the
 corresponding representative spatial patterns, i.e. the directions
 along which such points are prominently distributed. 
 The key element in our procedure is the expansion of the cumulant function that 
 can provide information beyond the first two moments. 
By maximizing the cumulant function \citep{Kenney51} over growing hyperspheres in
the data  space, a set of  components can be derived.
We first apply our procedure to three types of multivariate random vectors
(Normal, Skew-Normal, Gamma).
 Then we demonstrate in the general case that
  if a direction exists for which the marginal probability density of projected data display a larger tail than in all other  directions, then our procedure is able to select that direction.

When finding the first component, we show that PCA is a special case of our approach  whenever the Gaussian assumption is satisfied.
Besides the Gaussian case we show that, for any probability density, the
 solution derived from the centered cumulant function can be transformed to the first principal
component by decreasing the radius of the hypersphere.
Other principal components could be found as well, by a generalization
of the proposed method.
Our method is computationally cheap, and the solutions are found in
the form of unit (normalized) vectors, as in the case of PCA, allowing
a unidimensional projection with an easy geometrical interpretation. 

In summary, this paper focuses on the problem of characterizing spatial
patterns associated to large anomalies, i.e. large deviations from the
sample mean, when the data set under study cannot be assumed to be  normally distributed.

\section{Maximizing the cumulant function}
\label{sec: cumul fct}
In the univariate case, the cumulant function of the random variable $X$ with finite moments is defined as the following scalar function 
$$
 \log \Big\{ \Esp \big[ \exp( s X)\big]\Big\}= \sum_{n=1}^{\infty}\kappa_n   { s^n \over n!},
$$
 where $s \in \reels$, $\Esp(.)$ represents the mean function and the scalar $\kappa_n$ corresponds to the $n^{th}$ cumulant of $X$. 

The first two cumulants $\kappa_1$ and $\kappa_2$ are simply the mean
and the variance of $X$, respectively. 
The third and fourth cumulants are classically called 
the skewness and the kurtosis parameters.
Concerning the existence of cumulants, we assume in this paper that all the cumulant coefficients are finite and that the cumulant function is always well defined.
The cumulant function and its coefficients have many interesting properties. 
For example, if $X$ and $Y$ are  two independent random variables then 
the $n^{th}$ cumulant of the sum $X+Y$ is equal to the sum of 
the $n^{th}$ cumulant of  $X$ and the $n^{th}$ cumulant of  $Y$ for any integers $n$. 
If $X$ follows a Gaussian distribution, then all but the first two cumulants are equal to zero. 
In a multivariate framework, the cumulant function of the random vector 
${\bf X} = (X_1, \dots, X_m)^t$  is simply defined as 
\begin{equation}\label{eq: G def}
%G_{\bf X}({\mathbf s}) = 
\log \Big\{ \Esp \big[ \exp( {\bf s}^t {\bf X})\big]\Big\}
\mbox{, for all  ${\mathbf s}^t =(s_1, \dots, s_m) \in \reels^m$.}
\end{equation}
As in the univariate case, the linear and Gaussian properties associated to the cumulant function defined by 
(\ref{eq: G def}) still hold, but the cumulant coefficients formulas  are more cumbersome to write down in a multivariate framework.
For more information about cumulants, we refer the reader to \cite{Kenney51}.

To identify possible favorite projection directions with respect to the multivariate cumulant function, we first  rewrite the vector ${\mathbf s} =(s_1, \dots, s_m)^t$ in 
  (\ref{eq: G def}) as 
the  product  
$
{\mathbf s} = |{\mathbf s}| \times  \mbox{\boldmath \(\theta\)}
$, 
where $ |{\mathbf s}|^2 = \sum s_i^2$, the scalar $ |{\mathbf s}|$ represents the norm (``radius") of ${\mathbf s}$ 
and $\mbox{\boldmath \(\theta\)}$ is the unit ``angular/direction" vector defined as 
$\theta_i = s_i / |{\mathbf s}|$
(note that $\mbox{\boldmath \(\theta\)}^t \mbox{\boldmath \(\theta\)}=1$). 
Secondly,  the cumulant function  for our vector ${\bf X} = (X_1, \dots, X_m)^t$ projected  along
the direction  vector $\mbox{\boldmath \(\theta\)}$ is introduced  as 
\begin{eqnarray}\label{eq: projG def}
G_{|{\mathbf s}| }(\mbox{\boldmath \(\theta\)}) &=& 
\log \Big\{ \Esp \big[ \exp( |{\mathbf s}| \;  {\mbox{\boldmath \(\theta\)}}^t
  {\bf X}) \big] \Big\}, \nonumber\\
&=& \sum_{n=1}^\infty k_n(\mbox{\boldmath \(\theta\)})\frac{ |{\mathbf s}|^n}{n!}.
\end{eqnarray}
Our algorithmic strategy is to maximize the cumulant function at fixed
non-small $|{\mathbf s}|$ with respect
to the angular component $\mbox{\boldmath \(\theta\)}$ that varies over an
unit hypersphere.
Practically, we have to find the optimal $\mbox{\boldmath \(\theta\)}_s$ directions defined by 
\begin{equation}\label{eq: max pb}
\mbox{\boldmath \(\theta\)}_s = \mbox{argmax} \left[  
G_{|{\mathbf s}| }(\mbox{\boldmath \(\theta\)}) \mbox{ such that }
{\mbox{\boldmath \(\theta\)}^t\mbox{\boldmath \(\theta\)}=1}
\right],
\mbox{ for any $|{\mathbf s}|$.}
\end{equation}
If  the radius $|{\mathbf s}|$ is small enough, and the mean
$k_1$ is zero,
the variance $k_2(\mbox{\boldmath \(\theta\)})$ dominates the centered cumulant function  
and  the contributions of other cumulants can be
neglected.
In this situation, finding the first  PCA component can be viewed as a
special case of this optimization procedure,
because maximizing the cumulant function for small  $|{\mathbf s}|$ is equivalent to maximize  the
variance. 
As the value of $|{\mathbf s}|$ grows, higher and higher order cumulants become more  dominant.
Our main goal is to find  
$\mbox{\boldmath \(\theta\)}_s$ in (\ref{eq: max pb}) for the largest admissible $|{\mathbf s}|$ and study their properties. 
We will call the solutions of such an optimization scheme the {\it Maxima of the Cumulant Function} (MCF) directions.

We anticipate that, since the scalar product ${\bf \theta}^t {\bf X}$ is invariant under orthogonal
transformations, the cumulant function is invariant as well.
Given that the unit hypersphere is also invariant, our algorithm is symmetric
respect to orthogonal
transformations.
For instance, if data vectors are rotated by a given angle, the solutions
of the algorithm, in terms of $\mbox{\boldmath \(\theta\)}$, are rotated by the same amount.
This symmetry implies that if the probability density is isotropic, then the
cumulant function is isotropic as well, and no relative maxima of $G$ exist on
the unit hypersphere.
In that case no directions are selected: for a rotationally symmetric
distribution there is indeed no
preferred direction along which anomalies are prominent, they are
distributed uniformly over all angles.

To illustrate  our optimization procedure we derive explicit cumulant
maximization schemes for three special cases of multivariate  family
distributions, see Section \ref{sec: theory}.
For assessing the outputs  of our algorithm when applied to real data, we refer the reader to the second part of this paper  \cite{Bernacchia07}).

% Finally, in Section \ref{sec: mmdt}, we demonstrate that for any
% multivariate distribution, {\bf if a direction exist for which the marginal
%   probability density of projected data display the fattest tail, respect to a
%   local angle displacement, that direction is a solution of our algorithm.
% Hence, our algorithm provide the directions along which anomalies are mostly expected.}

\section{Theoretical examples}
\label{sec: theory}
To study the properties of the
maximization method developed in Section \ref{sec: cumul fct}, three examples
of distribution functions are considered in this paper.
We choose these three families because  explicit results can be
derived and they have been classically used in the statistical modeling of
temperatures and precipitation data.

Without loss of generality, some relevant matrices are diagonal  thereafter.
However, since the solutions are invariant respect to orthogonal
transformations, they may be rotated along with the corresponding coordinate change.
Analytical calculations are performed for the general multivariate case, while
figures are given for the bivariate case.
 
\subsection{Multivariate Gaussian vectors}
Suppose that the  data at hand can be appropriately fitted  by a multivariate Gaussian vector.
We assume that the observations have been centered (zero) mean and we denote
the covariance matrix as $\Sigma$.
The cumulant function of the centered Gaussian vector \citep[e.g.,][]{Kenney51} is equal to 
$$
\log \Big\{ \Esp \big[ \exp( {\bf s}^t {\bf X})\big]\Big\} = 
 {1 \over 2} {\bf s}^t \Sigma {\bf s}  .
$$
Hence, it is easy to show that all cumulants but the second
are equal to zero. 
The decomposition in Equation (\ref{eq: projG def}), 
${\mathbf s} = |{\mathbf s}| \times  \mbox{\boldmath
  \(\theta\)}$, implies that the cumulant function
becomes 
%\begin{equation}
%\label{genfungau}
$$
G_{|{\mathbf s}| }(\mbox{\boldmath \(\theta\)})=
\frac{|{\mathbf s}|^2}{2}k_2(\mbox{\boldmath \(\theta\)})=
\frac{|{\mathbf s}|^2}{2} \; \mbox{\boldmath \(\theta\)}^t {\bf \Sigma}\mbox{\boldmath \(\theta\)}
$$
Our optimization problem is to maximize 
$G_{|{\mathbf s}| }(\mbox{\boldmath \(\theta\)})$ under the constraint $\mbox{\boldmath \(\theta\)}^t\mbox{\boldmath \(\theta\)}=1$. 
To  find the optimal $\mbox{\boldmath \(\theta\)}_s$ defined by 
(\ref{eq: max pb}), we introduce a function 
to be maximized, constrained by the Lagrange multiplier $\lambda$, as 
$$ 
\frac{|{\mathbf s}|^2}{2} \; \mbox{\boldmath \(\theta\)}^t {\bf \Sigma}\mbox{\boldmath \(\theta\)}
- \lambda (\mbox{\boldmath \(\theta\)}^t\mbox{\boldmath \(\theta\)}-1).
$$
Setting the gradient with respect to $\mbox{\boldmath \(\theta\)}$ to zero gives 
\begin{equation}\label{eq : eigen}
2 \frac{|{\mathbf s}|^2}{2}  {\bf \Sigma}\mbox{\boldmath \(\theta\)}- 2 \lambda  \mbox{\boldmath \(\theta\)} =0.
\end{equation}
Let  $\lambda_{\Sigma} $   be  the largest eigenvalue of the covariance matrix ${\bf \Sigma}$   and $\mbox{\boldmath \(\theta\)}_{\Sigma} $ its associated eigenvector.
Introducing $\lambda_s = {|{\mathbf s}|^{2} \over 2}\lambda_{\Sigma}$, we can  write 
$
{|{\mathbf s}|^2 \over 2} {\bf \Sigma}\mbox{\boldmath \(\theta\)}_{\Sigma}=  \lambda_s  \mbox{\boldmath \(\theta\)}_{\Sigma}
$.
Consequently, $\mbox{\boldmath \(\theta\)}_{\Sigma} $ is the solution of (\ref{eq : eigen}).
Note that $\mbox{\boldmath \(\theta\)}_{\Sigma} $  depends on $\Sigma$ but
not on $|\bf s|$.
The optimal direction, for the Gaussian case, is $\mbox{\boldmath
  \(\theta\)}_s=\mbox{\boldmath \(\theta\)}_{\Sigma} $, and it corresponds to the classical first principal component.

To illustrate this result, a bivariate vector of the normal distribution is presented in Figure \ref{nfig}
(in which a contour plot is drawn in logarithmic scale). The matrix $\bf \Sigma$ is assumed to be diagonal, with entries 1.2 and 0.5143.
%%%%%%%%%%%%%%%%%%%%%%%%%%%%%%%
\begin{figure}
% \centerline{ \psfig{figure=Nfig.eps, width=16cm, angle=0}}
\includegraphics[scale=.65]{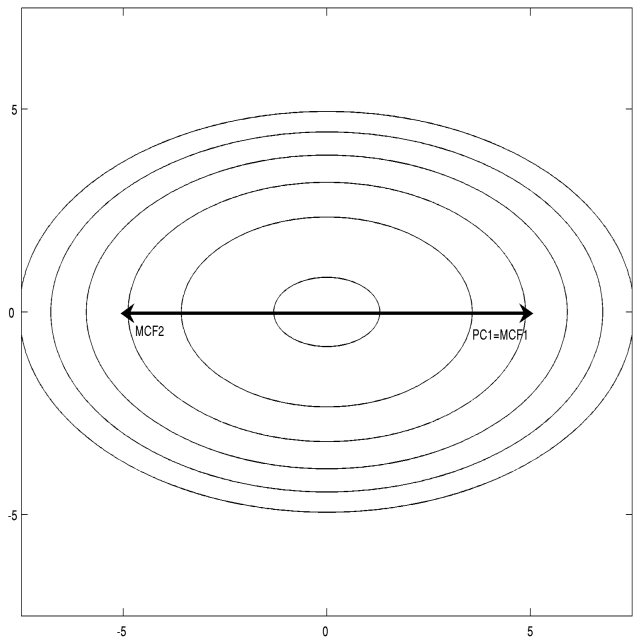}
\caption{Isoprobability contours of the multivariate
  Gaussian distribution, with zero mean and variances $1.2$ and $0.5143$
  (entries of the diagonal covariance matrix). The first principal component is shown (PC1), together
  with the two (opposite) maxima of the cumulant function (MCF1 and MCF2). All vectors
  are in arbitrary scale.
The maxima of cumulant function are parallel to the first principal component,
all pointing towards the large anomalies, in terms of high
probability (at fixed vector norm).
 Probability contours are $10^{-1}$,$10^{-3}$,$10^{-5}\ldots 10^{-11}$}
\label{nfig}
\end{figure}
%%%%%%%%%%%%%%%%%%%%%%%%%%%%%%%
The first principal component (PC1), corresponding to the eigenvalue 1.2 is horizontal, while the second principal component, corresponding to the eigenvalue 0.5143 is vertical, and is not displayed.
The two maxima of the cumulant function (MCF1 and MCF2) are just the positive and negative part of PC1.
Indeed, both $\mbox{\boldmath \(\theta\)}_{\Sigma}$ and $-\mbox{\boldmath
    \(\theta\)}_{\Sigma}$ are solutions of (\ref{eq : eigen}).
 While this is
  always true for PCA, the maxima of the cumulant function may  in general neither be 
  parallel nor orthogonal.
 
PC1 is indeed the direction along which large anomalies are ditributed in the
Gaussian case.
In order to derive PC2 and other higher principal components  from
the cumulant function, one would have  to determine not only its maxima, but also its minima and
saddle points.

%%%%%%%%%%%%%%%%%%%%%%%%%%%%%%%
\subsection{Multivariate Skew-Normal vectors}
To introduce skewness to the Gaussian density, while keeping  some of valuable properties of the normal distribution, Azzalini and his co-authors 
\citep[e.g][]{Azzalini96, Azzalini99,Gonzalez04} have extended the normal density to a larger class called   the  Skew-Normal (SN) density, that is defined as 
\begin{equation}
\label{SNdist}
f(\mathbf{x})=2\phi_\Sigma(\mathbf{x})\Phi(\mbox{\boldmath $\alpha$}^t\mathbf{x})
\end{equation}
where $\phi$ is a multivariate Normal probability density function with zero mean and
covariance matrix $\Sigma$, and $\Phi$ is the cumulative density function of an
univariate Gaussian random variable with zero mean and unit variance. 
The vector {\boldmath  $\alpha$} corresponds to the degree of skewness. 
When {\boldmath  $\alpha=0$}, there is no skewness, and the SN distribution
reduces to the Gaussian case. 
From (\ref{SNdist}), it is possible to derive 
the cumulant function  \citep[e.g][]{Azzalini96} of a SN vector 
$$
\log \Big\{ \Esp \big[ \exp( {\bf s}^t {\bf X})\big]\Big\} = 
 {1 \over 2} {\bf s}^t \Sigma {\bf s} + \log \left( 2 \Phi({\scriptstyle \sqrt{\frac{\pi}{2}}} 
\mbox{\boldmath $\mu$}^t {\bf s}) \right).
$$
where  $\mbox{\boldmath $\mu$}$  represents the mean vector, and is equal to 
$$
\mbox{\boldmath $\mu$}=\frac{\Sigma\mbox{\boldmath $\alpha$}}{\sqrt{\frac{\pi}{2}(1+\mbox{\boldmath $\alpha$}^t\Sigma\mbox{\boldmath $\alpha$})}}.
$$
Note that the covariance matrix of the SN distribution is not $\Sigma$ 
but $\Sigma-\mbox{\boldmath $\mu$}\mbox{\boldmath $\mu$}^t$  
\citep{Azzalini99}.

A bivariate example of the SN distribution is presented in Figure \ref{snfig}, in which a contour plot is drawn in logarithmic scale. The matrix  $\Sigma$  is chosen to be diagonal with entries $1.2$ and $0.5143$. 
The skewness vector $\alpha$ is taken to be equal to $(4.365,-1.455)$.
The distribution in Figure \ref{snfig} has been centered such that the mean vector is zero. In the bivariate example, {\boldmath $\mu$} is $(0.8367,-0.1195)$.
%%%%%%%%%%%%%%%%%%%%%%%%%%%%%%%%%%%%%
\begin{figure}
% \centerline{ \psfig{figure=SNfig.eps, width=16cm, angle=0}}
\includegraphics[scale=.65]{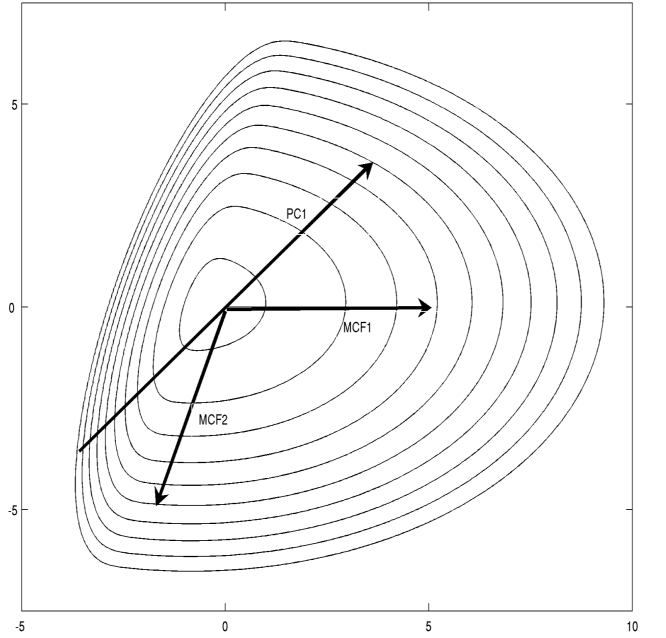}
\caption{Isoprobability contours of the multivariate
  Skew-Normal distribution, with parameters $\alpha=(4.365,-1.455)$ and
 the matrix  $\Sigma$  is chosen to be diagonal with entries $1.2$ and
 $0.5143$. 
The first principal component is shown (PC1), together
  with the two maxima of the cumulant function (MCF1 and MCF2). All vectors
  are in arbitrary scale.
In this case, the two maxima of the cumulant function are not related to
the first principal component, and point towards the (local) large anomalies, in terms of high
probability (at fixed, and large, vector norm).
Probability contours are $10^{-1}$,$10^{-3}$,$10^{-5}\ldots 10^{-19}$}
\label{snfig}
\end{figure}
%%%%%%%%%%%%%%%%%%%%%%%%%%%%%%%%%%%%%

From the SN cumulant function \citep{Azzalini99}, we can write  the cumulant function of the centered vector $\mbox{\boldmath \(\theta\)}^t({\bf X}- \mbox{\boldmath $\mu$}$)  as 
\begin{equation}
\label{genfunSN}
G_{|{\mathbf s}| }(\mbox{\boldmath \(\theta\)})=
-|{\mathbf s}|\mbox{\boldmath $\mu$}^t\mbox{\boldmath \(\theta\)}+
{|{\mathbf s}|^2 \over 2}\mbox{\boldmath \(\theta\)}^t \Sigma\mbox{\boldmath \(\theta\)}+\mbox{log}\Big[2\Phi({\scriptstyle \sqrt{\frac{\pi}{2}}}|{\mathbf s}|\mbox{\boldmath $\mu$}^t\mbox{\boldmath \(\theta\)})\Big]
\end{equation}
While it is not possible to find explicit solutions of the maximization problem defined by 
(\ref{eq: max pb}) for (\ref{genfunSN}), one can  provide valuable approximated  solutions  for both
 small  and large $|{\mathbf s}|$. 
In the former case,  the following two Taylor expansions are the key elements to derive our results
\begin{equation}\label{log(1+s)}
\mbox{log}(1+s) =  1 + s - \frac{s^2}{2}+o(s^3) \mbox{ and } 
\end{equation}
\begin{equation}\label{erf}
1+\mbox{erf}(s) =  1+\;\frac{2s}{\sqrt{\pi}}+o(s^3)
\end{equation}
where erf corresponds to  the error function defined by \\
$$
2\Phi(s)=1+\mbox{erf}\Big(\frac{s}{\sqrt{2}}\Big).
$$
Then we can write the following approximation 
\begin{eqnarray*}
\mbox{log}\Big[2\Phi({\scriptstyle \sqrt{\frac{\pi}{2}}}|{\mathbf
  s}|\mbox{\boldmath $\mu$}^t\mbox{\boldmath \(\theta\)})\Big]
  &=& \mbox{log}\Big[1+\mbox{erf}({\scriptstyle \frac{\sqrt{\pi}}{2}}|{\mathbf
  s}|\mbox{\boldmath $\mu$}^t\mbox{\boldmath \(\theta\)})\Big]  \\
  &  =  & \mbox{log}\Big[1+|{\mathbf
  s}|\mbox{\boldmath $\mu$}^t\mbox{\boldmath \(\theta\)}+o(|{\mathbf
  s}|^3)\Big] \mbox{, by (\ref{log(1+s)})}, \\
  & = & |{\mathbf
  s}|\mbox{\boldmath $\mu$}^t\mbox{\boldmath \(\theta\)}-{|{\mathbf s}|^2
  \over 2}\mbox{\boldmath \(\theta\)}^t\mbox{\boldmath $\mu$}\mbox{\boldmath
  $\mu$}^t\mbox{\boldmath \(\theta\)}+o(|{\mathbf s}|^3) \mbox{, by (\ref{erf})}.
\end{eqnarray*}
From   Equation (\ref{genfunSN}), it follows  that  the cumulant function 
$G_{|{\mathbf s}| }(\mbox{\boldmath \(\theta\)})$ 
is approximately equal to 
\begin{equation}
\label{genfunSNsmalls}
G_{|{\mathbf s}| }(\mbox{\boldmath \(\theta\)})\simeq{|{\mathbf s}|^2 \over 2}\mbox{\boldmath \(\theta\)}^t (\Sigma-\mbox{\boldmath $\mu$}\mbox{\boldmath $\mu$}^t)\mbox{\boldmath \(\theta\)}
\mbox{, for small $|{\mathbf s}|$.}
\end{equation}
As previously noticed,   the matrix $\Sigma-${\boldmath $\mu\mu$}$^t$ represents the
  covariance of the SN distribution \cite{Azzalini99}.
  Hence, the maximization of  the right hand side  of (\ref{genfunSNsmalls}) is equivalent to solving the system 
  defined by (\ref{eq : eigen}), but instead of working with $\Sigma$, we just need to replace $\Sigma$ by 
  $\Sigma-${\boldmath $\mu\mu$}$^t$ in (\ref{eq : eigen}). 
  Consequently, the solution  to maximize the SN cumulant function in the neighborhood of zero is  the largest eigenvector of the matrix $\Sigma-${\boldmath $\mu\mu$}$^t$, i.e. the PC1 of the SN covariance matrix.
  
For large $|{\mathbf s}|$, this result does not hold and different directions are obtained. 
We need to recall   the asymptotic expansion of the error
  function
$$
1+\mbox{erf}(s) \simeq 
\left\{
        \begin{array}{ll}
                -\frac{1}{s\;\sqrt{\pi}}\;\mbox{exp}(-s^2)\Big[1+o(s^{-3})\Big] & \mbox{, as $s\downarrow-\infty$,} \\
                2  & \mbox{,  as $s\uparrow+\infty$.}
        \end{array} 
        \right.
$$
If $\mbox{\boldmath
  $\mu$}^t\mbox{\boldmath \(\theta\)}<0$, the
logarithm in Equation (\ref{genfunSN}) can be expanded as
\begin{eqnarray*}
\mbox{log}\Big[2\Phi({\scriptstyle \sqrt{\frac{\pi}{2}}}|{\mathbf
  s}|\mbox{\boldmath $\mu$}^t\mbox{\boldmath \(\theta\)})\Big]
  &=&\mbox{log}\Big[1+\mbox{erf}({\scriptstyle \frac{\sqrt{\pi}}{2}}|{\mathbf
  s}|\mbox{\boldmath $\mu$}^t\mbox{\boldmath \(\theta\)})\Big]\\
&\simeq& -\frac{\pi}{2}\frac{|{\mathbf
  s}|^2}{2}\mbox{\boldmath \(\theta\)}^t \mbox{\boldmath $\mu$}\mbox{\boldmath
$\mu$}^t\mbox{\boldmath \(\theta\)}+O\big(\mbox{log}(|{\mathbf s}|^{-1})\big).
\end{eqnarray*}
In this case, we have 
$
G_{|{\mathbf s}| }(\mbox{\boldmath \(\theta\)})\simeq
{|{\mathbf s}|^2 \over 2}\mbox{\boldmath \(\theta\)}^t \Big(\Sigma-\frac{\pi}{2}\mbox{\boldmath $\mu$}\mbox{\boldmath $\mu$}^t\Big)\mbox{\boldmath \(\theta\)}.
$
The  direction that maximizes  $G_{|{\mathbf s}| }(\mbox{\boldmath
  \(\theta\)})$  is again a eigenvector, but this time is the eigevector of the matrix
$\Sigma-\frac{\pi}{2}\mbox{\boldmath $\mu$}\mbox{\boldmath $\mu$}^t$, 
pointing towards {\boldmath $\mu$}$^t\mbox{\boldmath \(\theta\)} < 0$.

If $\mbox{\boldmath
  $\mu$}^t\mbox{\boldmath \(\theta\)} \geq 0$, we have
  $$
\mbox{log}\Big[2\Phi({\scriptstyle \sqrt{\frac{\pi}{2}}}|{\mathbf
  s}|\mbox{\boldmath $\mu$}^t\mbox{\boldmath \(\theta\)})\Big]\simeq
   \left\{ \begin{array}{ll}
   \log 2 & \mbox{, $\mbox{\boldmath $\mu$}^t\mbox{\boldmath \(\theta\)} > 0$},\\
   0 & \mbox{, if $\mbox{\boldmath $\mu$}^t\mbox{\boldmath \(\theta\)}=0$}.\\
  \end{array} \right.
  $$
Then, the  cumulant function $G_{|{\mathbf s}| }(\mbox{\boldmath \(\theta\)})$ can be approximated by 
$
G_{|{\mathbf s}| }(\mbox{\boldmath \(\theta\)})\simeq
	{|{\mathbf s}|^2 \over 2}\mbox{\boldmath \(\theta\)}^t\Sigma\mbox{\boldmath \(\theta\)}
$
and maximizes by the largest eigenvector of $\Sigma$, pointing towards {\boldmath $\mu$}$^t\mbox{\boldmath \(\theta\)}\geq 0$. 

In summary, depending on the size of  $|{\mathbf s}|$ (small or large) and the sign 
{\boldmath $\mu$}$^t\mbox{\boldmath \(\theta\)}$ (positive or negative), 
the solutions  of Equation (\ref{eq: max pb}) for the SN distribution can be viewed as the largest eigenvectors of three different matrices, 
$\Sigma- ${\boldmath $\mu$}{\boldmath $\mu$}$^t$, $\Sigma$ and $\Sigma-\frac{\pi}{2}${\boldmath $\mu$}{\boldmath $\mu$}$^t$.

For the bivariate example of Figure \ref{snfig}, the PC1 is shown (in arbitrary scale), explaining 60\% of the variance. The second PC (40\% of the variance) is orthogonal to PC1 and is not displayed. 
Both maxima of the cumulant function for large $|{\mathbf s}|$, denoted as
MCF1 and MCF2 (respectively for {\boldmath $\mu$}$^t\mbox{\boldmath
  \(\theta\)}\geq0$ and $\mu$$^t\mbox{\boldmath
  \(\theta\)}<0$ ), are presented in Figure \ref{snfig} for the bivariate example (same scale as PC1).
The two local maxima point towards the large anomalies of the distribution:
this can be seen by noting that a point at the upper-right end of PC1 corresponds
to a small probability, and hence is less likely to be found, than a point at
the right end of MCF1 (the two points being of equal norm). 
Similarly, a point at the down-left end of PC1 is less likely, in probability, than a point at the down end of MCF2.

%%%%%%%%%%%%%%%%%%%%%%%%%%%%%%%
\subsection{Multivariate Gamma vectors}

This section investigates the multivariate Gamma distribution 
 defined by Cheriyan and Ramabhadran \citep[see][]{Kotz98}. 
 Each component of the data vector $\mathbf{X}$ is distributed following a Gamma distribution, and the components depend each other by means of an auxiliary variable $z$. The joint distribution is
\begin{equation}
\label{GAMpdf}
f(\mathbf{x})=\int g(z;\alpha_0)\prod_{i=1}^n g(x_i-z;\alpha_i) dz
\end{equation}
where $g$ is a gamma distribution, i.e.
$$
g(z,\alpha)=\frac{e^{-z}z^{\alpha-1}}{\Gamma(\alpha)}
$$
for $z\geq0$, equal to zero otherwise.
A bivariate example is presented in Figure \ref{gfig}, with $n=2$, $\alpha_0=2, \alpha_1=0.5$
and $\alpha_2=4$ in (\ref{GAMpdf}).
%%%%%%%%%%%%%%%%%%%%%%%%%%%%%%%
\begin{figure}[ht]
% \centerline{ \psfig{figure=GAMfig.eps, width=16cm, angle=0}}
\includegraphics[scale=.27]{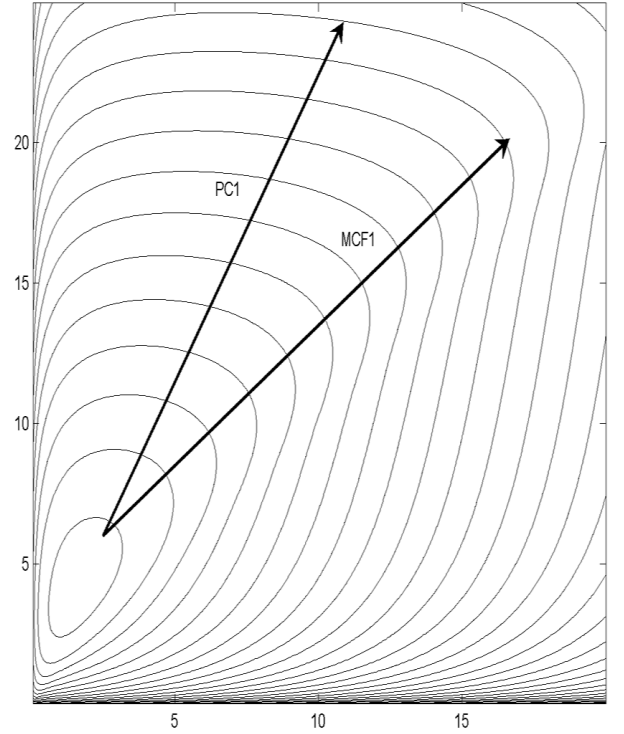}
\caption{Isoprobability contours of the multivariate
  Gamma distribution, with parameters $\alpha_0=2$, $\alpha_1=0.5$, $\alpha_2=4$.
 The first principal component is shown (PC1), together
  with the maximum of the cumulant function (MCF1). All vectors
  are in arbitrary scale.
Again, the maximum of the cumulant function is different from the first
principal component, and points towards the large anomalies, in terms of high
probability (at fixed, and large, vector norm).
Probability contours are $10^{-1/2}$,$10^{-1}$,$10^{-3/2}\ldots$}
\label{gfig}
\end{figure}
%%%%%%%%%%%%%%%%%%%%%%%%%%%%%%%

The cumulant function for this multivariate Gamma distribution can be written 
\citep{Kotz98} as 
$$
\log \Big\{ \Esp \big[ \exp( {\bf s}^t {\bf
  X})\big]\Big\}=-\alpha_0\mbox{log}\Big(1-\sum_{i=1}^n s_i\Big) 
  - \sum_{i=1}^n\alpha_i\mbox{log}(1-s_i)
$$ 
and the mean of the $i^{th}$ component
is $\mu_i = \alpha_0+\alpha_i$.
By replacing    ${\mathbf s} $ by $ |{\mathbf s}| \times  \mbox{\boldmath
  \(\theta\)}$,   the cumulant function of the centered vector $\mbox{\boldmath \(\theta\)}^t ({\bf X}- \mbox{\boldmath $\mu$})$ can be  written as 
\begin{eqnarray}
\label{genfunGAM}
G_{|{\mathbf s}| }(\mbox{\boldmath
  \(\theta\)})&=&-\alpha_0\mbox{log}\Big(1-|{\mathbf s}|\sum_{i=1}^n \theta_i\Big) \nonumber \\ 
  &-& \sum_{i=1}^n\alpha_i\mbox{log}(1-|{\mathbf s}|\theta_i)-|{\mathbf s}|\sum_{i=1}^n (\alpha_0+\alpha_i)\theta_i .
\end{eqnarray}
For small $|{\mathbf s}|$, the logarithms are approximated by the
  truncated Taylor expansions, i.e. 
$$
\sum_{i=1}^n\alpha_i\mbox{log}\big(1-|{\mathbf s}|\theta_i\big)
\simeq
-|{\mathbf
  s}|\sum_{i=1}^n\alpha_i\theta_i+\frac{|{\mathbf s}|^2}{2}\sum_{i=1}^n\alpha_i\theta_i^2
$$
and 
$$
\alpha_0\mbox{log}\big(1-|{\mathbf s}|\sum_{i=1}^n
  \theta_i\big)\simeq  -|{\mathbf s}|\sum_{i=1}^n
  \alpha_0\theta_i+\frac{|{\mathbf s}|^2}{2}\alpha_0\Big(\sum_{i=1}^n
  \theta_i\Big)^2.
$$
It follows that the cumulant function can be approximated by 
\begin{equation}
\label{genfunGAMsmalls}
G_{|{\mathbf s}| }(\mbox{\boldmath \(\theta\)})\simeq{|{\mathbf s}|^2 \over 2}\mbox{\boldmath \(\theta\)}^t C\mbox{\boldmath \(\theta\)}, 
\end{equation}
where the covariance matrix $C$ is defined by \cite{Kotz98}
$$
C=\left( \begin{array}{lll}\alpha_0+\alpha_1& & \alpha_0 \\ & \ddots & \\ \alpha_0 & & \alpha_0+\alpha_n \end{array} \right).
$$
Hence, for small $|{\mathbf s}|$, the solution of (\ref{eq: max pb}) is the classical PC1  eigenvector of the covariance matrix $C$ associated with the largest eigenvalue. It is plotted in Figure \ref{gfig} for the bivariate example. The negative part of PC1 is not displayed, as well as the second principal component which is just orthogonal to the first.

For large $|{\mathbf s}|$, the cumulant function is not defined, since the
logarithms in Equation (\ref{genfunGAM}) must have positive arguments, for all
unit vectors $\mbox{\boldmath \(\theta\)}$.
In particular, the following two inequalities have to be satisfied
$$
|{\mathbf s}|\theta_i<1
\mbox{ and }
|{\mathbf s}|\sum_{i=1}^n\theta_i<1.
$$
In other words, 
$$
|{\mathbf s}| <   \min \left( \min \left( {1 \over \theta_i} \right), {1 \over\sum \theta_i} \right).
$$
However, we take the largest allowed value of $|{\mathbf s}|$, which is
considered as a valid limit.
Since $\mbox{\boldmath \(\theta\)}$ is a unit vector, the maximum of each
component $\theta_i$ is $1$, which holds when all the other components are zero, while the maximum for the sum $\sum \theta_i$ is $\sqrt{n}$, which holds when $\theta_1=\theta_2=\ldots=\theta_n=1/\sqrt{n}$. 
The largest allowed value of $|{\mathbf s}|$ is then $1/\sqrt{n}$: in that case $G$ remains finite for all $\mbox{\boldmath \(\theta\)}$'s, except for $\theta_1=\theta_2=\ldots=\theta_n=1/\sqrt{n}$, where it diverges due to the first logarithm of Equation (\ref{genfunGAM}) (all the others remain finite).
For larger values of $|{\mathbf s}|$, $G$ diverges over subspaces larger than a single point.
Hence the boundary case $|{\mathbf s}|=1/\sqrt{n}$ is taken as representative for a ``large'' $|{\mathbf s}|$ limit, and the point $\theta_1=\theta_2=\ldots=\theta_n=1/\sqrt{n}$ is taken as the maximum of the cumulant function.

For the bivariate example, the maximum is plotted in Figure \ref{gfig}, denoted as MCF1, and scaled to PC1.
Note that for high values of the probability distribution (e.g. the smallest contour), PC1 seems a representative direction of the egg-like shape of the distribution, but for low probabilities it becomes clear that the line $\theta_1=\theta_2$ is responsible for the large deviations.
A point at the end of PC1 is indeed less probable than a point at the end of MCF1.

This result is understood by noting that the joint density (\ref{GAMpdf}) is the probability distribution of variables $x_i$ defined as $x_i=z_0+z_i$, where the variables $z_0,z_1,\ldots,z_n$ are independent and Gamma distributed with parameters $\alpha_0,\alpha_1,\ldots,\alpha_n$.
Hence, a large deviation of $z_0$, which occurs independently on others $z$'s,
corresponds to a large deviation of $\mathbf{x}$ which is placed on average
along the line $x_1=x_2=\ldots=x_n$ (that corresponds to
$\theta_1=\theta_2=\ldots=\theta_n$ for the cumulant function).

%%%%%%%%%%%%%%%%%%%%%%%%%%%%%%%%%%%%%%%%%
\section{Maximizing the marginal density tail}
\label{sec: mmdt}
We have seen that the multivariate cumulant function reduces to the variance
if the radius $|{\mathbf s}|$ tends to zero.
In that case, its maxima corresponds to the first principal component of data set.
If $|{\mathbf s}|$ grows, higher order cumulants come into play, but is not
clear what the corresponding maxima represent.
In order to clarify this point, we rewrite the cumulant function defined by 
(\ref{eq: projG def})  in terms of
the explicit integral over the probability density
$$
G_{|{\mathbf s}|}(\mbox{\boldmath \(\theta\)})
% =\log\Big\{\Esp\big[\exp(|{\mathbf s}|\; \mbox{\boldmath \(\theta\)}^t{\mathbf X})\big]\Big\}
=\log\int_{\reels^n} f({\mathbf x})\exp(|{\mathbf s}|\; \mbox{\boldmath \(\theta\)}^t{\mathbf x}) d{\mathbf x}
$$
This expression can be reduced to an unidimensional integral, by defining the
projected data as $Z=\mbox{\boldmath \(\theta\)}^t{\mathbf X}$, and the marginal probability density of the
projected data, $f_{\mbox{\boldmath \(\theta\)}}(z)$.
The cumulant function is then

\begin{equation}
G_{|{\mathbf s}|}(\mbox{\boldmath \(\theta\)})=\log\int_{-\infty}^{+\infty} 
f_{\mbox{\boldmath \(\theta\)}}(z)\exp(|{\mathbf s}|z) dz,
\end{equation}
which corresponds to the cumulant function of the univariate 
vector $Z=\mbox{\boldmath \(\theta\)}^t{\mathbf X}$ with distribution density $f_{\mbox{\boldmath \(\theta\)}}(z)$.
In the light of this representation, our maximization procedure is better understood:
we are looking for directions {\boldmath \(\theta\)} that correspond to a
marginal probability density $f_{\mbox{\boldmath \(\theta\)}}(z)$ displaying
maximal cumulant function, at fixed $|{\mathbf s}|$.
We want to demonstrate that if $|{\mathbf s}|$ grows, a larger cumulant function corresponds to a marginal
density with a fatter tail.
Hence our procedure
selects the directions corresponding to the marginal densities with fatter tails, where the
anomalous behaviour is expected.

Specifically, consider two different directions {\boldmath \(\theta\)} and {\boldmath
  \(\theta\)}$'$: we want to demonstrate that if the marginal distribution
along {\boldmath \(\theta\)} has a fatter tail than the distribution along
{\boldmath \(\theta\)}$'$, then the cumulant function has also a fatter tail along
{\boldmath \(\theta\)} respect to {\boldmath \(\theta\)}$'$.
More formally, we have the following theorem.
\newtheorem{theorem}{Theorem}
\begin{theorem}
Let ${\mbox{\boldmath \(\theta\)}}$ and ${\mbox{\boldmath \(\theta\)}'}$ be two directions. 
If there exists a real $z^*$ such that the density distribution $f_{\mbox{\boldmath \(\theta\)}}(z)$
of the random variable 
$\mbox{\boldmath \(\theta\)}^t{\mathbf X}$
 is strictly larger than the density   $f_{\mbox{\boldmath \(\theta\)}'}(z)$ for all $z>z^*$, i.e.
 $$
 f_{\mbox{\boldmath \(\theta\)}}(z)>f_{\mbox{\boldmath
    \(\theta\)}'}(z),  \;\;  \mbox{ for all $z>z^*$}, 
 $$
then there exists a radius $|{\mathbf s}|^*$ such the cumulant function of 
$\mbox{\boldmath \(\theta\)}^t{\mathbf X}$ 
and $\mbox{\boldmath \(\theta\)}'^t{\mathbf X}$
satisfies 
$$
G_{|{\mathbf    s}|}(\mbox{\boldmath \(\theta\)})>G_{|{\mathbf s}|}(\mbox{\boldmath
  \(\theta\)}'), \;\;   \mbox{ for all $ |{\mathbf s}|>|{\mathbf s}|^*$.}
$$
\end{theorem}

\begin{proof}
In order to prove the result, we start by noting that the following inequality holds
$$
\exp(|{\mathbf s}|z)\big[f_{\mbox{\boldmath \(\theta\)}}(z)-f_{\mbox{\boldmath \(\theta\)}'}(z)\big]
>  \exp(|{\mathbf s}|z^*)\big[f_{\mbox{\boldmath \(\theta\)}}(z)-f_{\mbox{\boldmath \(\theta\)}'}(z)\big]
$$
for all $  z>z^*$,
where we have replaced the exponential function with its minimum value in the interval $z\in(z^*,+\infty)$.
The inequality holds because the density difference  $f_{\mbox{\boldmath
    \(\theta\)}}(z)-f_{\mbox{\boldmath \(\theta\)}'}(z)$ is positive in this interval, by assumption.
Since the above inequality holds in the whole interval $z\in(z^*,+\infty)$, it can be integrated over, i.e.
\begin{eqnarray*}
\int_{z^*}^{+\infty} \exp(|{\mathbf s}|z)\big[f_{\mbox{\boldmath \(\theta\)}}(z)-f_{\mbox{\boldmath \(\theta\)}'}(z)\big] dz > \\
\;\;\;\;\;\; \exp(|{\mathbf s}|z^*)\int_{z^*}^{+\infty} \big[f_{\mbox{\boldmath \(\theta\)}}(z)-f_{\mbox{\boldmath \(\theta\)}'}(z)\big] dz.
\end{eqnarray*}
Given that the two densities are normalized, i.e.

$$
\int_{-\infty}^{+\infty}f_{\mbox{\boldmath \(\theta\)}}(z)dz=\int_{-\infty}^{+\infty}f_{\mbox{\boldmath \(\theta\)}'}(z)dz=1
$$
we can rewrite the right hand
side (r.h.s.) of the last  inequality 
by inverting the integration order and the marginal densities as
\begin{eqnarray}
\exp(|{\mathbf s}|z^*)\int_{z^*}^{+\infty} \big[f_{\mbox{\boldmath
    \(\theta\)}}(z)-f_{\mbox{\boldmath \(\theta\)}'}(z)\big] dz=\nonumber \\ \exp(|{\mathbf s}|z^*)\int_{-\infty}^{z^*}\big[f_{\mbox{\boldmath \(\theta\)}'}(z)-f_{\mbox{\boldmath \(\theta\)}}(z)\big] dz\nonumber
\end{eqnarray}
Rearranging the four terms gives indeed the normalization condition times the exponential.
Note that since the integral in the left hand side (l.h.s.) is, by assumption, positive, the integral in the
r.h.s. must be positive as well, because they are multiplied by the same
positive constant $\exp(|{\mathbf s}|z^*)$.

Given the above inequalities, we only need to demonstrate that it exists an
$|{\mathbf s}|^*$ such that, for all $ |{\mathbf s}|>|{\mathbf s}|^*$,
\begin{eqnarray}
\label{in3}
\exp(|{\mathbf s}|z^*)\int_{-\infty}^{z^*} \big[f_{\mbox{\boldmath \(\theta\)}'}(z)-f_{\mbox{\boldmath \(\theta\)}}(z)\big] dz
>\\
\int_{-\infty}^{z^*}\exp(|{\mathbf s}|z)\big[f_{\mbox{\boldmath \(\theta\)}'}(z)-f_{\mbox{\boldmath \(\theta\)}}(z)\big] dz\nonumber 
\end{eqnarray}
where the value of $|{\mathbf s}|^*$ must be determined.
Note that even if the integral in the l.h.s. is positive,  the density difference
$f_{\mbox{\boldmath \(\theta\)}'}(z)-f_{\mbox{\boldmath \(\theta\)}}(z)$  is not guaranteed to be
positive for all $z\in (-\infty,z^*)$.
If the integrand was positive as well, Equation (\ref{in3}) would hold
trivially for all
values of $|{\mathbf s}|$, because the exponential in the l.h.s. is always
larger than that of r.h.s.
This corresponds to the case in which, beside $f_{\mbox{\boldmath \(\theta\)}}(z)>f_{\mbox{\boldmath
    \(\theta\)}'}(z)\;\;\;\forall z>z^*$, we assume also $f_{\mbox{\boldmath \(\theta\)}}(z)<f_{\mbox{\boldmath
    \(\theta\)}'}(z)\;\;\;\forall z<z^*$.
However, we do not need this additional request, and we just note that it
would help our procedure by leaving the problem of using a large $|{\mathbf s}|$.

Since the exponential is positive, we can rewrite Equation (\ref{in3}) as
\begin{eqnarray*}
\int_{-\infty}^{z^*}\big[f_{\mbox{\boldmath \(\theta\)}'}(z)-f_{\mbox{\boldmath \(\theta\)}}(z)\big] dz
>\\
\int_{-\infty}^{z^*}\exp\big(|{\mathbf s}|(z-z^*)\big)\big[f_{\mbox{\boldmath \(\theta\)}'}(z)-f_{\mbox{\boldmath \(\theta\)}}(z)\big] dz
\end{eqnarray*}
for all $|{\mathbf s}|>|{\mathbf s}|^*$.
The integral in the l.h.s. is positive, as stated above, and is
independent on $|{\mathbf s}|$, while the integral in the r.h.s. converges to
zero for $|{\mathbf s}|\longrightarrow+\infty$, as long as the density
difference remains finite, because the exponential tends to zero in the whole
interval $z\in (-\infty,z^*)$.
If we define $|{\mathbf s}|^*$ as the largest possible value of $|{\mathbf s}|$
for which the two integrals are equal, then for all $|{\mathbf s}|>|{\mathbf
  s}|^*$ the integral in the l.h.s. is larger than that of r.h.s.

Finally, we have for all $|{\mathbf s}|>|{\mathbf s}|^*$
\begin{eqnarray*}
\int_{z^*}^{+\infty}dz\;\exp(|{\mathbf s}|z)\big[f_{\mbox{\boldmath \(\theta\)}}(z)-f_{\mbox{\boldmath \(\theta\)}'}(z)\big]
>\\
\int_{-\infty}^{z^*}dz\;\exp(|{\mathbf s}|z)\big[f_{\mbox{\boldmath \(\theta\)}'}(z)-f_{\mbox{\boldmath \(\theta\)}}(z)\big].
\end{eqnarray*}
By rearranging terms, this is equivalent to $G_{|{\mathbf
    s}|}(\mbox{\boldmath \(\theta\)})>G_{|{\mathbf s}|}(\mbox{\boldmath
  \(\theta\)}')$, and the theorem is demonstrated.

\end{proof}

\section*{Discussion}

In this paper, we have introduced a novel method selecting the
spatial patterns representative for the large deviations in
the dataset. 
The method consists in finding the vectors in the space of data for
which the cumulant function is maximal.
As in the case of PCA, the spatial patterns are found as
normalized directions in the space of data, and a linear projection can be
performed, with an easy geometrical
interpretation.
If one is interested on the large deviations, the projection allows to safely
perform Extreme Value Analysis (\cite{Coles01}). 
In both cases the subspaces are ordered:
in PCA the order follows the fraction of variance of each
subspace; The maxima of the cumulant function are ordered by the 
value of $G$. 
However, while PCA accounts for the mass of the distribution, the
cumulant function can cover the full range of the distribution.

Principal components are always symmetric, while large anomalous patterns, 
if generated by nonlinear processes, are expected to be neither specular nor orthogonal.
Accordingly, the maxima of the cumulant function are not necessarily symmetric, since they account
for the whole structure of dependencies, and not only covariances.
Other nonsymmetric techniques, such as oblique Varimax rotations, generally depend on the 
frame of reference: vector solutions do not covary with the space of data under orthogonal
transformations. 
Hence, solutions depend not only on the shape of the underlying probability distribution, 
but also on its orientation: an undesirable property for our purposes.
The maxima of the cumulant function, instead, vary with the probability
density whose shape is the only feature determining the maxima.

In the case of normally distributed data, the maximization of
 cumulant function yields the first principal component for all values
 of $|s|$: the elliptically symmetric distribution is characterized by the two tails along 
the major axis of the ellipse, i.e. the first principal component.
When the method is applied to Skew-Normal and Gamma distributions, for
non-small $|s|$, the maxima of the cumulant function determine large anomalies: 
high probability directions far from the center of mass.
Note that the limit radius $|s|\rightarrow\infty$ is the innovative key from a technical point of view, 
allowing for analytical solutions.
Using the limit, we were also able to demonstrate that the solutions of our
algorithm correspond, in general, to the directions along which the marginal
probability density display the fattest tails.

Using the cumulant function is computationally cheap, there is no
 free parameter, and has the advantage of searching for local solutions, all of which are of interest.
When a solution is found, is always a good solution, in contrast with
 neural networks applications, for which it must be questioned if it
 is the global or just a local solution.
In real applications \citep[see][]{Bernacchia07}, the radius $|s|$ must be taken
as large as possible, until the expected error in
the estimate of the cumulant function, 
due to the finite sample, reach a tolerance value \citep[see][]{Bernacchia07}.
This corresponds to maximize a combination of cumulants which is 
of the highest reliable order with the given amount of data, 
accounting for the available set of anomalies.

The solutions of our algorithm are expected to transform continuously as the
radius $|s|$ varies.
Hence, even if the limit  $|s|\rightarrow\infty$ cannot be taken in practice,
the solutions for a finite value of $|s|$ are expected to represent a substantial
departure from the PCA solution, towards the formal solution at
$|s|\rightarrow\infty$.
From the theoretical point of view, future work could be devoted to studying
more in detail the nature of solutions at varying $|s|$.
For instance, one could attempt to find under which conditions and to what
extent the solutions are in between the PCA solution and the formal solution at infinite $|s|$.
From the applicative point of view, we expect several datasets 
 \citep[e.g.][]{Bernacchia07} to be
fruitfully analyzed with our new method.
 
Note that the logarithm is taken for illustrative purposes: the moment
generating function could be used instead of the cumulant
function, since the maximization is invariant under application of
a monotonous function.
The present definition is however confortable in avoiding extremely large numbers.
Centering of data about the mean is also a practical step, related with the constraint
of dealing with finite samples: if the limit $|s|\rightarrow\infty$ could be really 
taken, the mean would be irrelevant.

Results of our procedure are corrupted if variables are standardized by a rescaling,
since the relative scale of different directions is the key in detecting
anomalies and comparing the size of tails.
If variables are standardized, our procedure reduce to a special case of
Independent Component Analyisis (ICA, see \cite{Hyvarinen07}), detecting independent components rather
than large anomalies.
Results are also corrupted if we try to get an empirical estimate of the
cumulant function when the underlying probability density decays less than
exponentially fast.
In that case, the cumulant function diverge, and the variance of the empirical estimate
increases with the size of the sample (\cite{Sornette00}), implying that the estimate is always unreliable.

\begin{acknowledgements}
 This work  was supported by the european E2-C2 grant, the National Science Foundation (grant: NSF-GMC (ATM-0327936)), by The Weather and Climate Impact Assessment Science Initiative 
at the  National Center for Atmospheric Research 
(NCAR) and the ANR-AssimilEx project.
 The authors would also like to credit the contributors of the R project. 
Finally, we would like to thank Isabella Bordi, Marcello Petitta and Alfonso Sutera for valuable discussions.
\end{acknowledgements}

%\newpage

%% If figures and tables must be numbered 1a, 1b, etc. the following command
%% should be inserted before the begin{} command.

\addtocounter{figure}{-1}\renewcommand{\thefigure}{\arabic{figure}a}

\end{document}